\documentclass[12pt,reqno]{amsart}

\usepackage{multirow}
\usepackage{hhline}

\usepackage{mathtools}
\usepackage{times}
\usepackage[T1]{fontenc}
\usepackage{mathrsfs}
\usepackage{latexsym}
\usepackage[dvips]{graphics}
\usepackage[titletoc, title]{appendix}
\usepackage{amsmath,amsfonts,amsthm,amssymb,amscd}
\usepackage[dvipsnames]{xcolor}
\usepackage{hyperref}
\usepackage{amsmath}

\usepackage{color}
\usepackage{breakurl}

\usepackage{comment}
\newcommand{\bburl}[1]{\textcolor{blue}{\url{#1}}}

\newtheorem{thm}{Theorem}[section]

\newtheorem{cor}[thm]{Corollary}

\newtheorem{lem}[thm]{Lemma}
\newtheorem{prop}[thm]{Proposition}

\newtheorem{defi}[thm]{Definition}
\newtheorem{rek}[thm]{Remark}

\DeclareMathOperator{\supp}{supp}
\DeclareMathOperator{\spann}{span}
\DeclareMathOperator{\sgn}{sgn}
\numberwithin{equation}{section}

\DeclareFontFamily{U}{mathx}{}
\DeclareFontShape{U}{mathx}{m}{n}{<-> mathx10}{}
\DeclareSymbolFont{mathx}{U}{mathx}{m}{n}
\DeclareMathAccent{\widehat}{0}{mathx}{"70}
\DeclareMathAccent{\widecheck}{0}{mathx}{"71}

\begin{document}

\title[The TGA and $f$-bounded approximations]{The Thresholding Greedy Algorithm versus Approximations with Sizes Bounded by Certain Functions $f$}

\author{H\`ung Vi\d{\^e}t Chu}

\email{\textcolor{blue}{\href{mailto:hungchu2@illinois.edu}{hungchu2@illinois.edu}}}
\address{Department of Mathematics, University of Illinois at Urbana-Champaign, Urbana, IL 61820, USA}

\begin{abstract}
Let $X$ be a Banach space and $(e_n)_{n=1}^\infty$ be a basis. For a function $f$ in a large collection $\mathcal{F}$ 
(closed under composition), we define and characterize $f$-greedy and $f$-almost greedy bases. 
We study relations among these bases as $f$ varies and show that while a basis is not almost greedy, it can be $f$-greedy for some $f\in \mathcal{F}$. Furthermore, we prove that for all non-identity function $f\in \mathcal{F}$, we have the surprising equivalence
$$f\mbox{-greedy}\ \Longleftrightarrow \ f\mbox{-almost greedy}.$$
We give various examples of Banach spaces to illustrate our results. 
\end{abstract}



\thanks{The author is thankful to Timur Oikhberg for helpful feedback on earlier drafts of this paper and for showing the author that for certain functions $f$, a quasi-greedy and $f$-disjoint superdemocratic basis is $f$-unconditional.}

\maketitle

\tableofcontents

\section{Introduction}

\subsection{Background and motivation}

Let $X$ be an infinite-dimensional Banach space (with dual $X^*$) over the field $\mathbb{F} = \{\mathbb{R}, \mathbb{C}\}$. In this paper, we work with a general basis, which is a countable collection of vectors $(e_n)_{n=1}^\infty\subset X$ such that: i) $\overline{\spann(e_n)_{n=1}^\infty} = X$; ii) there exists a collection $(e_n^*)_{n=1}^\infty\subset X^*$, called the biorthogonal functionals, that satisfy $e^*_n(e_m) = \delta_{n, m}$; iii) $\sup_{n} \|e_n\|\|e_n^*\|_{*} < \infty$ and semi-normalization: $$0 \ <\ \inf_{n} \|e_n\|\ \leqslant\ \sup_{n}\|e_n\|\ <\ \infty.$$ 
M-boundedness and semi-normalization imply the existence of $\mathbf c_1, \mathbf c_2 > 0$ such that 
$$0 \ <\ \mathbf c_1 \ := \ \inf_n\{\|e_n\|, \|e_n^*\|\}\ \leqslant \ \sup_n\{\|e_n\|, \|e_n^*\|\}\ := \ \mathbf c_2 \ <\ \infty.$$ One important goal of approximation theory is to efficiently approximate each vector $x\in X$ by a finite linear combination of vectors in a given basis. In 1999, Konyagin and Temlyakov \cite{KT1} introduced a natural, order-free method of approximation, called the Thresholding Greedy Algorithm (TGA). The algorithm chooses the largest coefficients (in modulus) of each vector $x$ in its representation with respect to a basis $(e_n)_{n=1}^{\infty}$. Formally, a set $A\subset\mathbb{N}$ is a greedy set of a vector $x$ of order $m$ if 
$$|A| \ =\ m\mbox{ and }\min_{n\in A}|e_n^*(x)| \ \geqslant\ \max_{n\notin A}|e_n^*(x)|.$$
Let $G(x, m)$ denote the set of all greedy sets of $x$ of order $m$. For each $x\in X$, the TGA thus produces a sequence of approximants $(\sum_{n\in \Lambda_m}e_n^*(x)e_n)_{m=1}^\infty$, where $\Lambda_m\in G(x,m)$.

Konyagin and Temlyakov \cite{KT1} defined greedy and quasi-greedy bases. For a finite set $A\subset\mathbb{N}$, let $P_{A}(x) = \sum_{n\in A}e_n^*(x)e_n$ and $P_{A^c}(x) = x - P_A(x)$. Also, let $\mathbb{N}^{<\infty}$ denote the collection of all finite subsets of $\mathbb{N}$.
\begin{defi}\normalfont
A basis is said to be greedy if there exists a constant $\mathbf C\geqslant 1$ such that
$$\|x-P_{\Lambda}(x)\| \ \leqslant\ \mathbf C\sigma_m(x),\forall x\in X, \forall m\in \mathbb{N}, \forall \Lambda\in G(x, m),$$
where $\sigma_m(x) = \inf_{\substack{|A|\leqslant m\\ (a_n)\subset \mathbb F}}\|x - \sum_{n\in A}a_n e_n\|$.
\end{defi}

\begin{defi}\normalfont
A basis $(e_n)_{n=1}^\infty$ is said to be quasi-greedy\footnote{This definition of the quasi-greedy property is equivalent to the convergence $$P_{\Lambda_m}(x)\ \longrightarrow\ x, \forall x\in X, \forall \Lambda_m\in G(x, m).$$ For details, see \cite[Theorem 4.1]{AABW}.} if there exists a constant $\mathbf C> 0$ such that
\begin{equation}\label{e3}\|x-P_{\Lambda}(x)\|\ \leqslant\ \mathbf C\|x\|, \forall x\in X, \forall m\in\mathbb{N}, \forall \Lambda\in G(x, m).\end{equation}
In this case, we let $\mathbf C_\ell$ denote the smallest constant $\mathbf C$ for \eqref{e3} to hold, and we say that 
$(e_n)_{n=1}^\infty$ is $\mathbf C_\ell$-suppression quasi-greedy. Also, let $\mathbf C_q$ be the smallest constant such that 
$$\|P_{\Lambda}(x)\|\ \leqslant\ \mathbf C_q\|x\|, \forall x\in X, \forall m\in\mathbb{N}, \forall \Lambda\in G(x, m).$$
We say that $(e_n)_{n=1}^\infty$ is $\mathbf C_q$-quasi-greedy. 
\end{defi}

A beautiful characterization of greedy bases involves unconditionality (Definition \ref{de1}) and democracy (Definition \ref{de2}).

\begin{thm}[Konyagin and Temlyakov \cite{KT1}]\label{ktc}
A basis is greedy if and only if it is unconditional and democratic. 
\end{thm}

Observe that in the definition of greedy bases, we compare the TGA, which involves projections onto greedy sets, against approximations with arbitrary coefficients. Hence, it is also natural to compare the TGA against projections only (without arbitrary coefficients). This thinking leads us to the so-called almost greedy bases, introduced by Dilworth et al. \cite{DKKT}.

\begin{defi}\normalfont
A basis is said to be almost greedy if there exists a constant $\mathbf C\geqslant 1$ such that
$$\|x-P_{\Lambda}(x)\| \ \leqslant\ \mathbf C\widetilde{\sigma}_m(x),\forall x\in X, \forall m\in \mathbb{N}, \forall \Lambda\in G(x, m),$$
where $\widetilde{\sigma}_m(x) = \inf_{|A| \leqslant m}\|x - P_A(x)\|$.
\end{defi}

It is worth noting that the requirement $``|A| \leqslant m"$ in the definition of $\widetilde{\sigma}_m(x)$ can be replaced by $``|A|= m"$, and the basis is still almost greedy with the same constant $\mathbf C$. While greedy bases are almost greedy, the converse does not necessarily hold. A famous example of an almost greedy basis that is not greedy is due to Konyagin and Temlyakov \cite{KT1} (also see \cite[Example 10.2.9]{AK}). Dilworth et al. \cite{DKKT} gave the following nice and surprising characterization of almost greedy bases. 

\begin{thm}\cite[Theorem 3.3]{DKKT}\label{dkktc}
Let $(e_n)_{n=1}^\infty$ be a basis of a Banach space. The following are equivalent
\begin{enumerate}
    \item [i)] $(e_n)_{n=1}^\infty$ is almost greedy.
    \item [ii)] $(e_n)_{n=1}^\infty$ is quasi-greedy and democratic.
    \item [iii)] For any (respectively, every) $\lambda > 1$, there is a constant $\mathbf C_\lambda > 0$ such that
    $$\|x-P_{\Lambda}(x)\| \ \leqslant\ \mathbf C_\lambda \sigma_m(x), \forall x\in X, \forall m\in\mathbb{N}, \forall \Lambda\in G(x, \lceil\lambda m\rceil).$$
\end{enumerate}
\end{thm}

The equivalence between i) and ii) shares the same spirit as Theorem \ref{ktc}, while the equivalence between i) and iii) shows that enlarging the greedy sums in the definition of greedy bases brings us to the strictly larger realm of almost greedy bases. In other words, for almost greedy bases, an approximation using greedy sums of size $\lceil \lambda m\rceil$ gives us essentially the smallest error term from using an arbitrary $m$-term approximation. 

Besides greedy and almost greedy bases, two other notable greedy-type bases are the so-called partially greedy (PG) (introduced by Dilworth et al. \cite{DKKT}) and reverse partially greedy (RPG) bases (introduced by Dilworth and Khurana \cite{DK}). They are not the focus of this paper; however, we would like to make a few remarks about these bases. Unlike (almost) greedy bases, the notion of both PG and RPG bases depends on the ordering of the basis. For example, in the definition of partially greedy bases, we compare the efficiency of the TGA against partial sums, the more convenient method of approximation. Recently, some authors expressed their concerns about the notion of partially greedy bases as it is order-dependent, while the TGA is order-free \cite[page 6]{AABBL}. While that is a well-founded concern, one can also argue that partially greedy bases are crucial because they give us information about how the TGA performs against the most convenient method of approximation, namely partial summations. In other words, there is no point in choosing the largest coefficients for our approximation if we cannot outperform partial summations. 

\subsection{Motivation}

Inspired by Theorem \ref{dkktc}, the author of the present paper \cite{C1, C2} investigated what happens if we enlarge the greedy sums (by a constant factor $\lambda > 1$) in the definition of almost greedy, PG, and RPG bases. Surprisingly, while doing so gives new, weaker greedy-type bases in the case of PG and RPG bases (see \cite[Theorem 1.5]{C1} and \cite[Theorem 1.4]{C2}), \cite[Theorem 5.4]{C1} states that we do not obtain a weaker greedy-type basis in the case of almost greedy bases. 
To put it another way, if the sizes of greedy sums and of the projective approximation terms in the definition of almost greedy bases are of the same order, i.e., $\lceil\lambda m\rceil$ and $m$, respectively, then we obtain only equivalences of being almost greedy. However, we suspect that such an equivalence no longer holds if the sizes are of different order, say $m$ and $\log m$, for example. This paper examines what happens when the sizes are of different order. 

Our main results not only give new greedy-type bases and an unexpected surprising equivalence among these bases but also provide a partial solution to the above-mentioned concern of some researchers regarding the order-dependent notion of PG and RPG bases. In particular, while some PG and PRG greedy bases are not almost greedy, they may satisfy a certain order-free condition, which involves resizing projective approximations in the definition of (almost) greedy bases.

\subsection{Main results}
Throughout this paper, let $f:\mathbb{R}_{\geqslant 0}\rightarrow \mathbb{R}_{\geqslant 0}$ satisfy 
\begin{enumerate}
    \item $f(0) = 0$, $f(1)\leqslant 1$, and $f(x) > 0$ for all $x > 0$,
    \item $f$ is continuous, increasing, and concave on $[0,\infty)$, and 
    \item $f$ is differentiable on $(0, \infty)$.
\end{enumerate}
Let $\mathcal{F}$ denote the set of all such functions. Note that $\mathcal{F}$ contains functions like $f(x) = cx^{\gamma}$ for $c, \gamma\in (0,1]$. Also, $\mathcal{F}$ is closed under composition.

\begin{defi}\normalfont
A basis $(e_n)_{n=1}^\infty$ is said to be $f$-almost greedy if there exists a constant $\mathbf C > 0$ such that 
\begin{equation}\label{e11}\|x-P_\Lambda(x)\|\ \leqslant\ \mathbf C\widetilde{\sigma}_{f(m)}(x),\forall x\in X, \forall m\in\mathbb{N},\forall \Lambda\in G(x, m).\end{equation}
The least $\mathbf C$ is denoted by $\mathbf C_{a, f}$. 
\end{defi}

\begin{defi}\normalfont
A basis $(e_n)_{n=1}^\infty$ is said to be $f$-greedy if there exists a constant $\mathbf C>0$ such that 
\begin{equation}\label{e30}\|x-P_{\Lambda}(x)\|\ \leqslant\ \mathbf C\sigma_{f(m)}(x),\forall x\in X, \forall m\in\mathbb{N},\forall \Lambda\in G(x,m).\end{equation}
The least such $\mathbf C$ is denoted by $\mathbf C_{g, f}$.
\end{defi}

At the first glance, these definitions only make sense for functions $f$ satisfying $f(m) < m + 1$ for all $m\in\mathbb{N}$. Otherwise, even the canonical basis of $c_0$ does not satisfy \eqref{e11}. For example, suppose that $f(3) = 4$ and consider the vector $x = (1, 1, 1, 1, 0, \ldots)$. Let $\Lambda = \{1, 2, 3\}$ be a greedy set of $x$. While $\|x-P_{\Lambda}(x)\| = 1$, $\widetilde{\sigma}_{f(3)}(x) = 0$. Therefore, there is no $\mathbf C$ for \eqref{e11} to hold. However, as we prove later, if $f\in \mathcal{F}$, then $f(m)\leqslant m$ for all $m\in\mathbb{N}$.

The main goal of this paper is to establish the following unexpected result.
\begin{thm}\label{m3}
Let $f$ be a non-identity function in $\mathcal{F}$. Then a basis is $f$-almost greedy if and only if it is $f$-greedy. 
\end{thm}
Theorem \ref{m3} says that if $f$ grows slower than linear functions, it does not matter whether we compare the TGA against projections or linear combinations of basis vectors. We shall prove Theorem \ref{m3} by characterizing $f$-greedy and $f$-almost greedy bases separately, then show that the two characterizations are equivalent when $f$ is non-identity. Below we list several interesting results we shall encounter along the way. See Definition \ref{d11} for Property (A, $f$) and Definition \ref{dd1} for $f$-unconditionality, a new concept that lies strictly between quasi-greediness and unconditionality. 

\begin{thm}\label{m1}
Let $(e_n)_{n=1}^\infty$ be a basis of a Banach space $X$. The following hold.
\begin{enumerate}
    \item[i)] If $(e_n)_{n=1}^\infty$ is $\mathbf C_{a, f}$-$f$-almost greedy, then $(e_n)_{n=1}^\infty$ is $\mathbf C_{a, f}$-suppression quasi-greedy and has $\mathbf C_{a, f}$-Property (A, $f$).
    \item[ii)] Conversely, if $(e_n)_{n=1}^\infty$ is $\mathbf C_\ell$-suppression quasi-greedy and has $\Delta_{A, f}$-Property (A, $f$), then $(e_n)_{n=1}^\infty$ is $\mathbf C_\ell \Delta_{A, f}$-$f$-almost greedy. 
\end{enumerate}
\end{thm}

\begin{thm}\label{mm2}
Let $(e_n)_{n=1}^\infty$ be a basis of a Banach space. The following hold.
\begin{enumerate}
    \item[i)] If $(e_n)_{n=1}^\infty$ is $\mathbf C_{g,f}$-$f$-greedy, then it is $\mathbf C_{g,f}$-suppression $f$-unconditional and has $\mathbf C_{g,f}$-Property (A, $f$).
    \item[ii)] If $(e_n)_{n=1}^\infty$ is $\mathbf K_{su}$-suppression $f$-unconditional and has $\Delta_{A, f}$-Property (A, $f$), then it is $\Delta_{A, f}(\mathbf K^2_{su} + 2\mathbf K_{su})$-$f$-greedy. 
\end{enumerate}
\end{thm}

While the proof of Theorem \ref{m1} follows standard argument from the literature, the proof of Theorem \ref{mm2}
item ii) requires a nontrivial deviation that relies on properties of $f\in \mathcal{F}$ and a new concept, called $f$-greedy sets. The next corollary involves (super)democracy (see Definition \ref{de3}).

\begin{cor}\label{c1}
Let $(e_n)_{n=1}^\infty$ be a basis of a Banach space $X$ and $f\in \mathcal{F}$ be non-identity. The following are equivalent
\begin{enumerate}
    \item[i)] $(e_n)_{n=1}^\infty$ is $f$-greedy.
    \item[ii)] $(e_n)_{n=1}^\infty$ is $f$-almost greedy.
    \item[iii)] $(e_n)_{n=1}^\infty$ is quasi-greedy and has Property (A, $f$).
    \item[iv)] $(e_n)_{n=1}^\infty$ is quasi-greedy and $f$-(disjoint) superdemocratic.
    \item[v)] $(e_n)_{n=1}^\infty$ is quasi-greedy and $f$-(disjoint) democratic. 
\end{enumerate}
Furthermore, we can replace the quasi-greedy requirement in each of the above statements by $f$-unconditionality. 
\end{cor}

Let us discuss an application of Corollary \ref{c1} to the space $\ell_p\oplus \ell_q$ ($1\leqslant p < q < \infty$). \`{E}del'\v{s}te\v{i}n and Wojtaszczyk \cite{EW} showed that any unconditional basis $\mathcal{B}$ of the direct sum $\ell_p\oplus \ell_q$ consists of two subsequences, one of which, denoted by $\mathcal{B}_1$, is a basis for $\ell_p$, while the
other, denoted by $\mathcal{B}_2$, is a basis for $\ell_q$. By the gliding hump argument, we can find subsequences of $\mathcal{B}_1$ and $\mathcal{B}_2$ that are equivalent to the canonical bases of $\ell_p$ and $\ell_q$, respectively. Therefore, none of the unconditional bases of $\ell_p\oplus \ell_q$ is democratic, so $\ell_p\oplus \ell_q$ has no greedy basis. However, $\ell_p\oplus \ell_q$ has an $x^{p/q}$-greedy basis. For example, let $(e_n)_{n=1}^\infty$ and $(f_n)_{n=1}^\infty$ be the canonical bases of $\ell_p$ and $\ell_q$, respectively. Define the basis $\mathcal{X} = (x_n)_{n=1}^\infty$, where $x_{2k-1} = (e_k, 0)$ and $x_{2k} = (0, f_k)$. It is easy to check that $\mathcal{X}$ is $x^{p/q}$-democratic. By Corollary \ref{c1}, $\mathcal{X}$ is $x^{p/q}$-greedy.

Another application of Corollary \ref{c1} is to the multivariate Haar system $\mathcal{H}^d_p$ ($d>1$) in $L_p([0,1]^d)$ for $1 < p < \infty$, which shall be discussed in Section \ref{ex}. Roughly speaking, though the system is not greedy, it is close to being greedy in the sense that $\mathcal{H}^d_p$ is $x^{1-\varepsilon}$-greedy for any $\varepsilon > 0$.

The paper is structured as follows: 
\begin{itemize}
\item Section \ref{prelim} discusses some properties of functions in $\mathcal{F}$ and of quasi-greedy bases. These results will be used throughout the paper.
\item Section \ref{fdemo} studies $f$-(disjoint) (super)democratic bases and Property (A, $f$). These are the first ingredient in characterizing $f$-(almost) greedy bases.
\item Section \ref{fcon} introduces and characterizes the new notion, $f$-unconditionality, which is the second ingredient in characterizing $f$-(almost) greedy bases. 
\item Section \ref{equiv} characterizes $f$-(almost) greedy bases and show the unexpected equivalence between $f$-greedy and $f$-almost greedy bases when $f$ is non-identity. 
\item Section \ref{Schreier} investigates the relationship among $f$-(almost) greedy bases when $f$ varies. 
\item Section \ref{ex} gives various examples to complement our results. 
\end{itemize}


\section{Preliminary results}\label{prelim}

Let us prove some important properties of functions in $\mathcal{F}$. 

\begin{prop}\label{p1}
Let $f\in \mathcal{F}$. The following hold
\begin{enumerate}
    \item[i)] If $0 < a\leqslant b$, then $\frac{f(a)}{a}\geqslant \frac{f(b)}{b}$.
    \item[ii)] For $x\geqslant 1$, $f(x)\leqslant x$. 
    \item[iii)] Let $m, n\in\mathbb{N}$ and $k$ be a nonnegative integer such that $n\geqslant k$. If $n\leqslant f(m)$, then $n-k\leqslant f(m-k)$.
\end{enumerate}
\end{prop}

\begin{proof}
i) Consider $g:\mathbb{R}_{>0}\rightarrow \mathbb{R}_{\geqslant 0}$ defined by $g(x) = f(x)/x$. It suffices to prove that 
\begin{equation}\label{e1}g'(x)\ =\ (f'(x)x-f(x))/x^2 \ \leqslant\ 0, \forall x > 0.\end{equation}
Recall that if $f$ is concave and differentiable on an interval $(u, v)$, then 
$$f(s)\ \leqslant\ f(t) + f'(t)(s-t),\forall s, t\in (u, v).$$
Therefore, for $0 < \varepsilon < x$, we have
$$f(\varepsilon) \ \leqslant \ f(x) + f'(x)(\varepsilon - x).$$
Letting $\varepsilon \rightarrow 0$, by continuity, we obtain 
$$0 \ \leqslant\ f(x) - xf'(x), \forall x > 0,$$
which implies \eqref{e1}.

ii) Let $x\geqslant 1$. By item i), $\frac{f(x)}{x}\leqslant \frac{f(1)}{1} \leqslant 1$. Hence, $f(x)\leqslant x$.

iii) Since $f(m)\leqslant m$ by item ii), we get $k\leqslant n\leqslant m$. If $n = k$, there is nothing to prove. If $n > k$, by item i), we have
$$\frac{f(m-k)}{f(m)}\ \geqslant\ \frac{m-k}{m} \ \geqslant\ \frac{n-k}{n}.$$
Hence, 
$$f(m-k)\ \geqslant\ f(m)\frac{n-k}{n}\ \geqslant \ n-k.$$
This completes our proof.
\end{proof}

\begin{prop}\label{pc}
If $f\in \mathcal{F}$, then either $f$ is the identity or there exists $\lambda > 1$ such that $f(x) \ \leqslant\ \frac{x}{\lambda}$ for all sufficiently large $x$.
\end{prop}

\begin{proof}
By Proposition \ref{p1} items i) and ii), we know that $\lim_{x\rightarrow \infty}\frac{f(x)}{x} =: s\leqslant 1$. If $s < 1$, then for sufficiently large $x$, $\frac{f(x)}{x} < \frac{s+1}{2} < 1$, which gives $f(x) < \frac{x}{2/(s+1)}$ and we are done. Suppose that $s = 1$. Since $\frac{f(x)}{x}$ is decreasing for $x > 0$ and $\lim_{x\rightarrow \infty}\frac{f(x)}{x} = 1$, it follows that $f(x)\geqslant x$ for $x > 0$. On the other hand, Proposition \ref{p1} item ii) states that $f(x)\leqslant x$ for $x\geqslant 1$. Therefore, we get $f(x) = x$ for all $x\geqslant 1$. It remains to show that $f(x) \leqslant x$ for all $x\in [0,1)$. Suppose, for a contradiction, that there exists a $p\in (0,1)$ such that $f(p) = p + \varepsilon$ for some $\varepsilon > 0$. Choose $q = 2-p$. We have
$$f(q) + f(p) \ =\ (2 - p) + (p + \varepsilon) \ =\ 2+\varepsilon \ >\ 2f(1)\ \Longrightarrow\ \frac{1}{2}f(q) + \frac{1}{2}f(p) \ >\ f\left(\frac{q}{2} + \frac{p}{2}\right),$$
which contradicts concavity. Therefore, if $s = 1$, then $f$ is the identity. 
\end{proof}

Next, we mention two well-known properties of a quasi-greedy basis. 
For $A\in \mathbb{N}^{<\infty}$, let $1_A := \sum_{n\in A}e_n$.

\begin{defi}\normalfont
A basis is said to have the UL property if there exist $\mathbf C_1, \mathbf C_2 > 0$ such that
$$\frac{1}{\mathbf C_1}\min|a_n|\|1_A\|\ \leqslant\ \left\|\sum_{n\in A}a_ne_n\right\|\ \leqslant\ \mathbf C_2\max|a_n|\|1_A\|,$$
for all $A\in \mathbb{N}^{<\infty}$ and scalars $(a_n)_{n\in A}\subset\mathbb{F}$.
\end{defi}

\begin{thm}\cite[Lemma 2.1]{DKKT}
If $(e_n)_{n=1}^\infty$ is $\mathbf C_\ell$-suppression quasi-greedy, then $(e_n)_{n=1}^\infty$ has the UL property. In particular, 
$$\frac{1}{2\mathbf C_\ell}\min|a_n|\|1_A\|\ \leqslant\ \left\|\sum_{n\in A}a_ne_n\right\|\ \leqslant\ 2\mathbf C_\ell\max|a_n|\|1_A\|,$$
for all $A\in \mathbb{N}^{<\infty}$ and scalars $(a_n)_{n\in A}\subset\mathbb{F}$.
\end{thm}

Another property of a quasi-greedy basis is the uniform boundedness of the truncation operators. Let $\alpha > 0$ and define $T_{\alpha}: X\rightarrow X$ by 
$$T_{\alpha}(x) \ =\ \sum_{n\in A_x}\alpha \sgn(e_n^*(x))e_n + P_{A_x^c}(x),$$
where $A_x = \{n: |e_n^*(x)| > \alpha\}$.

\begin{thm}\cite[Lemma 2.5]{BBG}\label{bto}
If $(e_n)_{n=1}^\infty$ is $\mathbf C_\ell$-suppression quasi-greedy, then $\|T_\alpha\|\leqslant \mathbf C_\ell$ for all $\alpha > 0$.
\end{thm}

The final result in this section follows from a well-known Lebesgue-type inequality. 
\begin{lem}\label{l3}
Let $(e_n)_{n=1}^\infty$ be a basis of a Banach space $X$. For each $N\in\mathbb{N}$, we have 
$$\|x-P_{\Lambda}(x)\|\ \leqslant\ (1 +  3\mathbf c_2^2N)\sigma_{m}(x),\forall x\in X, \forall m\leqslant N, \forall \Lambda\in G(x,m).$$
\end{lem}

\begin{proof}
The statement comes from \cite[Theorem 1.8]{BBG}.
\end{proof}

As an application of Lemma \ref{l3}, we have

\begin{prop}\label{ptu}
Let $f\in \mathcal{F}$ such that $f$ is not the identity. 
If $(e_n)_{n=1}^\infty$ is almost greedy, then $(e_n)_{n=1}^\infty$ is $f$-greedy.
\end{prop}

\begin{proof} Since $f$ is not the identity, by Proposition \ref{pc}, there exists $\lambda > 1$ such that $f(x)\leqslant \frac{x}{\lambda}$ for all sufficiently large $x$. 
As $(e_n)_{n=1}^\infty$ is almost greedy, Theorem \ref{dkktc} implies that $(e_n)_{n=1}^\infty$ is $x/\lambda$-greedy. For sufficiently large $m$, since $f(m)\leqslant \frac{m}{\lambda}$, we have $\sigma_{m/\lambda}(x)\leqslant \sigma_{f(m)}(x)$ for all $x\in X$, which gives \eqref{e30}. For small $m$, use Lemma \ref{l3}. 
\end{proof}

\begin{rek}\normalfont
The converse of Proposition \ref{ptu} is not true, as evidenced by unconditional bases of $\ell_p\oplus \ell_q$ ($1 \leqslant p < q <\infty$) and the multivariate Haar system in $L_p([0,1]^d)$ ($p, d > 1$). See the discussion after Corollary \ref{c1}.
\end{rek}

\section{About $f$-(disjoint) democracy and Property (A, $f$)}\label{fdemo}
We recall the democratic property of bases was due to Konyagin and Temlyakov \cite{KT1}.
\begin{defi}\normalfont\label{de2}
A basis is said to be democratic if there exists a constant $\mathbf C\geqslant 1$ such that 
$$\|1_A\|\ \leqslant\ \mathbf C\|1_B\|, \forall A, B\in\mathbb{N}^{<\infty}, |A|\leqslant |B|.$$
\end{defi}

A sign $\varepsilon$ is a sequence of scalars $(\varepsilon_n)_{n=1}^\infty \subset \mathbb{F}$ with modulus $1$. If $A\in \mathbb{N}^{<\infty}$, we let $1_{\varepsilon A} := \sum_{n\in A}\varepsilon_n e_n$. For $x\in X$ and $A, B\in \mathbb{N}^{<\infty}$, we write $x\sqcup A\sqcup B$ to mean that $\supp(x), A$, and $B$ are pairwise disjoint. Similarly, write $\sqcup_{i} A_i$ to mean that the sets $A_i$ are pairwise disjoint. Finally, $\|x\|_\infty:= \max_{n\in\mathbb{N}}|e_n^*(x)|$.

\begin{defi}\normalfont\label{de3}
A basis $(e_n)_{n=1}^\infty$ is said to be $f$-disjoint superdemocratic if there exists a constant $\mathbf C>0$ such that
\begin{equation}\label{e2}
    \|1_{\varepsilon A}\|\ \leqslant\ \mathbf C\|1_{\delta B}\|,
\end{equation}
for all disjoint $A, B\in\mathbb{N}^{<\infty}$ with $|A|\leqslant f(|B|)$ and for all signs $\varepsilon, \delta$. The least $\mathbf C$ that verifies \eqref{e2} is denoted by $\Delta_{sd, f}$. If $\varepsilon \equiv \delta \equiv 1$, we say that $(e_n)_{n=1}^\infty$ is $f$-disjoint democratic, and let $\Delta_{d, f}$ denote the corresponding constant. 
\end{defi}

\begin{rek}\normalfont
For the definition of $f$-(super)democratic, we drop the requirement that $A$ and $B$ are disjoint in Definition \ref{de3}. 
\end{rek}

\begin{prop}\label{p2}
Let $(e_n)_{n=1}^{\infty}$ be a quasi-greedy basis. Then $(e_n)_{n=1}^{\infty}$ is $f$-disjoint democratic if and only if it is $f$-disjoint superdemocratic. Similarly, a quasi-greedy basis is $f$-democratic if and only if it is $f$-superdemocratic.
\end{prop}

\begin{proof}
Let $(e_n)_{n=1}^{\infty}$ be a $\mathbf C_\ell$-suppression quasi-greedy basis.
By definition, if $(e_n)_{n=1}^\infty$ is $f$-disjoint superdemocratic, then it is $f$-disjoint democratic. Conversely, assume that $(e_n)_{n=1}^\infty$ is $\Delta_{d, f}$-$f$-disjoint democratic. Let $A, B\in \mathbb{N}^{<\infty}$ be disjoint and  $|A|\leqslant f(|B|)$. Pick signs $\varepsilon, \delta$. By the UL property, we have
$$\frac{1}{2\mathbf C_\ell}\|1_{\varepsilon A}\|\ \leqslant\ \|1_A\|\ \leqslant\ \Delta_{d, f}\|1_B\|\ \leqslant\ 2\Delta_{d, f}\mathbf C_\ell\|1_{\delta B}\|.$$
Therefore, $\|1_{\varepsilon A}\|\leqslant 4\Delta_{d, f}\mathbf C^2_\ell\|1_{\delta B}\|$.
\end{proof}

Finally, we introduce Property (A, $f$), which is closely related to the democratic property and will be used in due course. 

\begin{defi}\label{d11}\normalfont
A basis $(e_n)_{n=1}^\infty$ is said to have Property (A, $f$) if there exists a constant $\mathbf C\geqslant 1$ such that
\begin{equation}\label{e33}
    \|x + 1_{\varepsilon A}\|\ \leqslant\ \mathbf C\|x + 1_{\delta B}\|,
\end{equation}
for all $x\in X$ with $\|x\|_\infty\leqslant 1$, for all $A, B\subset\mathbb{N}^{<\infty}$ with $|A|\leqslant f(|B|)$ and $x\sqcup A\sqcup B$, and for all signs $\varepsilon, \delta$. The least such $\mathbf C$ is denoted by $\Delta_{A, f}$. 
\end{defi}

\begin{prop}\label{p22}
A basis $(e_n)_{n=1}^\infty$ has $\Delta_{A, f}$-Property (A, $f$) if and only if 
\begin{equation}\label{e12}\|x\|\ \leqslant\ \Delta_{A, f}\|x - P_A(x) + 1_{\varepsilon B}\|,\end{equation}
for all $x\in X$ with $\|x\|_\infty\leqslant 1$, for all $A, B\subset\mathbb{N}^{<\infty}$ such that $|A|\leqslant f(|B|)$ and $(x-P_A(x))\sqcup A\sqcup B$, and for all signs $\varepsilon$. 
\end{prop}

\begin{proof}
Assume that $(e_n)_{n=1}^\infty$ has $\Delta_{A, f}$-Property (A, $f$). Let $x, A, B, \varepsilon$ be chosen as in \eqref{e12}. We have
\begin{align*}
    \|x\|\ =\ \left\|x - P_A(x) + \sum_{n\in A}e_n^*(x)e_n\right\| &\ \leqslant\ \sup_{\delta}\left\|x-P_A(x) + 1_{\delta A}\right\|\\
    &\ \leqslant\ \Delta_{A, f}\left\|x-P_A(x) + 1_{\varepsilon B}\right\|.
\end{align*}
Conversely, let $x, A, B, \varepsilon, \delta$ be chosen as in Definition \ref{d11}. Let $y = x + 1_{\varepsilon A}$. We have
\begin{align*}
    \|x+1_{\varepsilon A}\|\ =\ \|y\|\ \leqslant\ \Delta_{A, f}\|y - P_A(y) + 1_{\delta B}\|\ =\ \Delta_{A, f}\|x + 1_{\delta B}\|.
\end{align*}
This completes our proof. 
\end{proof}

\begin{prop}\label{pe2}
Let $(e_n)_{n=1}^{\infty}$ be a quasi-greedy basis and $f\in \mathcal{F}$. Then the following are equivalent
\begin{enumerate}
\item[i)] $(e_n)_{n=1}^\infty$ is $f$-disjoint democratic.
\item[ii)] $(e_n)_{n=1}^\infty$ is $f$-democratic.
\item[iii)] $(e_n)_{n=1}^\infty$ is $f$-disjoint superdemocratic.
\item[iv)] $(e_n)_{n=1}^\infty$ is $f$-superdemocratic.
\item[v)] $(e_n)_{n=1}^\infty$ has Property (A, $f$).
\end{enumerate}
\end{prop}

\begin{proof}
In light of Proposition \ref{p2}, it suffices to show that for quasi-greedy bases,
\begin{itemize}
\item[a)] an $f$-disjoint democratic basis is $f$-democratic. 
\item[b)] Property (A, $f$) is equivalent to $f$-disjoint superdemocracy.
\end{itemize}

a) Let $A, B\in \mathbb{N}^{<\infty}$ and $|A|\leqslant f(|B|)$. By Proposition \ref{p1} item iii), $|A\backslash (A\cap B)|\leqslant f(|B\backslash (A\cap B)|)$. Hence, $f$-disjoint democracy and quasi-greediness imply that 
$$\|1_{A\backslash (A\cap B)}\|\ \leqslant\ \Delta_{d,f}\|1_{B\backslash (A\cap B)}\|\ \leqslant\ \Delta_{d,f}\mathbf C_q\|1_B\|\mbox{ and }\|1_{A\cap B}\|\ \leqslant\ \mathbf C_q\|1_B\|.$$
We, therefore, obtain
$$\|1_A\|\ \leqslant\ \|1_{A\backslash (A\cap B)}\|+\|1_{A\cap B}\| \ \leqslant\ \mathbf{C}_q(1+\Delta_{d,f})\|1_B\|.$$

b) Clearly, Property (A, $f$) implies $f$-disjoint superdemocracy. We prove the converse. Assume that $(e_n)_{n=1}^\infty$ is $\mathbf C_\ell$-suppression quasi-greedy and $\Delta_{sd, f}$-$f$-disjoint superdemocratic. We have
\begin{align*}\|x+1_{\varepsilon A}\|\ \leqslant\ \|x\| + \|1_{\varepsilon A}\|&\ \leqslant\ \mathbf C_\ell\|x+1_{\delta B}\| + \Delta_{sd, f}\|1_{\delta B}\|\\
&\ \leqslant\ \mathbf C_\ell \|x+1_{\delta B}\| + \Delta_{sd, f}(\mathbf C_{\ell}+1)\|x+1_{\delta B}\|\\
&\ =\ (\mathbf C_\ell + \Delta_{sd, f}(\mathbf C_{\ell}+1))\|x+1_{\delta B}\|.
\end{align*}\
Hence, $(e_n)_{n=1}^\infty$ has Property (A, $f$).
\end{proof}

\section{About $f$-unconditionality}\label{fcon}
Unconditionality is a well-known and strong requirement on bases. In this section, we introduce a weaker concept, called $f$-unconditionality, that shall be shown to lie strictly between quasi-greediness and unconditionality. We study this new property and use it to later characterize $f$-greedy bases. First of all, we have the classical definition of unconditionality. 

\begin{defi}\normalfont\label{de1}
A basis is said to be unconditional if there exists a constant $\mathbf C\geqslant 1$ such that
$$\|x-P_A(x)\|\ \leqslant\ \mathbf C\|x\|, \forall x\in X, \forall A\subset\mathbb{N}.$$
In this case, the least constant $\mathbf C$ is denoted by $\mathbf K_{su}$, called the suppression unconditional constant. 
\end{defi}

\begin{defi}\normalfont
Fix $f\in \mathcal{F}$. A set $A\in \mathbb{N}^{<\infty}$ is an $f$-greedy set of $x$ if $A$ can be written as the disjoint union of $E\sqcup F$, where $|E|\leqslant f(|A|)$, and 
$$|e_n^*(x)|\ \geqslant\ \|x-P_{A}(x)\|_\infty, \forall n\in F.$$
We call $(E, F)$ an $f$-greedy decomposition of $A$. Let $G(x, f)$ denote the collection of all $f$-greedy sets of $x$.
\end{defi}

\begin{prop}\label{pe1}
Fix $x\in X$ and $f\in \mathcal{F}$. A set $A\in \mathbb{N}^{<\infty}$ is an $f$-greedy set of $x$ if and only if there exists $B\subset A$ such that $B$ is a greedy set of $x$ of size at least $|A| - f(|A|)$.
\end{prop}

\begin{proof}Let $x\in X$ and $f\in\mathcal{F}$.
Assume that $A\in G(x, f)$. Let $(E, F)$ be an $f$-greedy decomposition of $A$. Then $|E|\leqslant f(|A|)$. If $F = \emptyset$, then $A = E$ and $|A| = |E| \leqslant f(|A|)$; that is, $|A| - f(|A|) = 0$ and there is nothing to prove. If $F\neq \emptyset$, let $\alpha = \min_{n\in F}|e_n^*(x)|$. Define $E' = \{n\in E: |e_n^*(x)|\geqslant \alpha\}$. It is easy to check that $E'\cup F$ is a greedy set of $x$. Furthermore, 
$$|E'\cup F| \ =\ |E'| + |F|\ =\ |A| - |E\backslash E'|\ \geqslant\ |A| - f(|A|).$$

Conversely, pick $A\in \mathbb{N}^{<\infty}$ such that there exists $B\subset A$ and $B$ is a greedy set of $x$ of size at least $|A| - f(|A|)$. We claim that $(A\backslash B, B)$ is an $f$-greedy decomposition of $A$. Indeed, $$\min_{n\in B}|e_n^*(x)|\ \geqslant\ \|x-P_{B}(x)\|_\infty\ \geqslant\ \|x-P_A(x)\|_\infty$$
and 
$$|A\backslash B| \ =\ |A| - |B| \ \leqslant\ |A| - (|A| - f(|A|))\ =\ f(|A|).$$
Therefore, $A$ is an $f$-greedy set of $x$.
\end{proof}

We now introduce $f$-unconditionality.

\begin{defi}\normalfont\label{dd1}
A basis $(e_n)_{n=1}^\infty$ is said to be $f$-unconditional if there exists a constant $\mathbf C\geqslant 1$ such that
$$\|x-P_A(x)\|\ \leqslant\ \mathbf C\|x\|,\forall x\in X, \forall A\in G(x, f).$$
The smallest such $\mathbf C$ is denoted by $\mathbf K_{su}$, and we say that $(e_n)_{n=1}^\infty$ is $\mathbf K_{su}$-suppression $f$-unconditional. In this case, let $\mathbf K_u$ denote the smallest constant such that 
$$\|P_A(x)\|\ \leqslant\ \mathbf K_{u}\|x\|,\forall x\in X, \forall A\in G(x, f),$$
and say that $(e_n)_{n=1}^\infty$ is $\mathbf K_u$-$f$-unconditional.
\end{defi}

\begin{rek}\normalfont We can obtain Theorem \ref{ktc} from Theorem \ref{mm2} and Proposition \ref{pe2} by choosing $f$ to be identity function. Our notion of $f$-unconditionality lies between quasi-greediness and unconditionality. By definition, we have the implications
$$\mbox{unconditional} \Longrightarrow f\mbox{-unconditional} \Longrightarrow \mbox{quasi-greedy}.$$
To show that our notion of $f$-unconditionality does not overlap with either of these notions, we show that there exists a conditional basis that is $x/\lambda$-unconditional for all $\lambda > 1$ (Example \ref{ex1c}), and there exists a quasi-greedy basis that is not $f$-unconditional for any unbounded $f\in\mathcal{F}$ (Example \ref{ex1a}).

Finally, observe that when $f$ is the identity function, $f$-unconditionality is the same as unconditionality. When $f$ is bounded, it is not hard to show that $f$-unconditionality is the same as quasi-greedy. 
\end{rek}

\begin{prop}\label{pp1}
If a basis $(e_n)_{n=1}^\infty$ is $\mathbf K_{su}$-suppression $f$-unconditional, then it is $\mathbf K_{su}$-suppression quasi-greedy. If a basis is $\mathbf K_u$-$f$-unconditional, then it is $\mathbf K_u$-quasi-greedy. 
\end{prop}

\begin{proof}
Simply note that for all $m\in\mathbb{N}$, $G(x, m)\subset G(x, f)$ by Proposition \ref{pe1}.
\end{proof}

We characterize $f$-unconditionality as follows. 

\begin{prop}\label{pd10}
Let $f\in \mathcal{F}$. A basis $(e_n)_{n=1}^\infty$ is $f$-unconditional if and only if it is quasi-greedy and there exists a constant $\mathbf C>0$ such that
\begin{equation}\label{e31}\|P_A(x)\|\ \leqslant\ \mathbf C\|x\|,\end{equation}
for all $x\in\mathbb{X}$ and for all $A\in\mathbb{N}^{<\infty}$ such that $|A|\leqslant f(|A| + |B|)$ for some greedy set $B$ of $x$ that is disjoint from $A$. 
\end{prop}

\begin{proof}
Assume that $(e_n)_{n=1}^\infty$ is $\mathbf K_u$-$f$-unconditional. By Proposition \ref{pp1}, $(e_n)_{n=1}^\infty$ is $\mathbf K_u$-quasi-greedy. To see \eqref{e31}, we fix $x\in\mathbb{X}$, a greedy set $B$ of $x$, and  $A\subset\mathbb{N}^{<\infty}$ such that $A\sqcup B$ and  $|A|\leqslant f(|A| + |B|)$. We have
$$\|P_{A}(x)\|\ \leqslant \ \|P_{A\cup B}(x)\| + \|P_{B}(x)\|\ \leqslant\ \|P_{A\cup B}(x)\| + \mathbf K_u\|x\|.$$
Since $A\cup B\in G(x, f)$, we get
$$\|P_{A\cup B}(x)\| \ \leqslant\ \mathbf K_u\|x\|.$$
We conclude that $\|P_A(x)\|\leqslant 2\mathbf K_u\|x\|$.

Conversely, assume that $(e_n)_{n=1}^\infty$ is $\mathbf C_q$-quasi-greedy and satisfies \eqref{e31}. Let $x\in X$ and $A\in G(x, f)$ with an $f$-greedy decomposition $(E, F)$. If $F = \emptyset$, then $A = E$ and $|E| \leqslant f(|E|)$. By \eqref{e31}, we obtain
$$\|P_A(x)\|\ \leqslant\ \mathbf C\|x\|.$$
If $F\neq\emptyset$, let $\alpha = \min_{n\in F}|e_n^*(x)|$ and $E' = \{n\in E: |e^*_n(x)| < \alpha\}$. Then $F' = F\cup (E\backslash E')$ is a greedy set of $x$.
We get 
$$\|P_{A}(x)\|\ \leqslant\ \|P_{E'}(x)\| + \|P_{F'}(x)\|\ \leqslant\ \|P_{E'}(x)\| + \mathbf C_{q}\|x\|.
$$
Furthermore, 
$$|E'|\ \leqslant\ |E|\ \leqslant\ f(|A|)\ =\ f(|E'| + |F'|).$$
Hence, \eqref{e31} implies that 
$\|P_{E'}(x)\| \leqslant \mathbf C\|x\|$. Therefore, 
$$\|P_{A}(x)\|\ \leqslant\ (\mathbf C_q + \mathbf C)\|x\|.$$
This completes our proof. 
\end{proof}

The above characterization inspires our introduction of Property ($f$).

\begin{defi}\label{d20}\normalfont
A basis $(e_n)_{n=1}^\infty$ is said to have Property ($f$) if there exists a constant $\mathbf C>0$ such that
$$\|P_A(x)\|\ \leqslant\ \mathbf C\left\|x + 1_{\varepsilon B}\right\|,$$
for all $x\in\mathbb{X}$ with $\|x\|_\infty\leqslant 1$, for sets $A, B\in\mathbb{N}^{<\infty}$ with $A\subset \supp(x)$, $x\sqcup B$, and $|A|\leqslant f(|A| + |B|)$, and for all signs $\varepsilon$.
\end{defi}

When $f$ is the identity function, Property ($f$) is the same as unconditionality.

\begin{prop}\label{ps}
A basis is $f$-unconditional if and only if it is quasi-greedy and has Property ($f$).
\end{prop}

\begin{proof}
Due to Propositions \ref{pp1} and \ref{pd10}, it suffices to prove that for a quasi-greedy basis, Property ($f$) is equivalent to \eqref{e31}. Assume \eqref{e31}. Let $x, A, B, \varepsilon$ as in Definition \ref{d20}. Set $y = x + 1_{\varepsilon B}$. Then $B$ is a greedy set of $y$ and $|A|\leqslant f(|A| + |B|)$. By \eqref{e31}, 
$$\|P_A(x)\|\ =\ \|P_A(y)\|\ \leqslant\ \mathbf C\|y\|\ =\ \mathbf C\|x+1_{\varepsilon B}\|.$$

Conversely, assume that $(e_n)_{n=1}^\infty$ has $\Delta$-Property ($f$) and is $\mathbf C_\ell$-suppression quasi-greedy. Choose $x, A, B$ as in \eqref{e31}. If $B = \emptyset$, then $|A|\leqslant f(|A|)$. By Property ($f$), $\|P_A(x)\|\leqslant \Delta\|x\|$. If $B\neq \emptyset$, let $\alpha = \min_{n\in B}|e_n^*(x)|$ and $\varepsilon = (\sgn(e_n^*(x))$.
By Theorem \ref{bto}, we have
$$\|P_{B^c}(x) + \alpha 1_{\varepsilon B}\|\ \leqslant\ \mathbf C_\ell\|x\|.$$
Let $A' = A\cap \supp(P_{B^c}(x))$, which satisfies $|A'|\leqslant f(|A'| + |B|)$ due to Proposition \ref{p1} item iii). By Property ($f$), 
$$\|P_A(x)\|\ =\ \|P_{A'}(P_{B^c}(x))\|\ \leqslant\ \Delta\|P_{B^c}(x) + \alpha 1_{\varepsilon B}\|.$$
Therefore, 
$$\|P_A(x)\|\ \leqslant\ \mathbf C_\ell\Delta \|x\|,$$
as desired. 
\end{proof}

\begin{rek}\normalfont
We provide several examples to show that our characterization is not redundant. In particular, there exists a quasi-greedy basis that does not have Property ($f$) for any unbounded $f$ (Example \ref{ex1a}). On the other hand, there exists a non-quasi-greedy basis that has Property ($x/\lambda$) for any $\lambda > 1$ (Example \ref{ex1b}). This is the best we can do since Property ($x$) is the same as unconditionality. 
\end{rek}


\section{Characterizations of $f$-(almost) greedy bases}\label{equiv}

In this section, we prove Theorems \ref{m1}, \ref{mm2} and Corollary \ref{c1}, all of which give equivalences of $f$-(almost) greedy bases.

\begin{proof}[Proof of Theorem \ref{m1}]
i) Assume that $(e_n)_{n=1}^\infty$ is $\mathbf C_{a, f}$-$f$-almost greedy. Since by definition, $\widetilde{\sigma}_m(x)\leqslant \|x\|$ for all $x\in X$ and $m\in\mathbb{N}$, \eqref{e11} implies that $(e_n)_{n=1}^\infty$ is $\mathbf C_{a, f}$-suppression quasi-greedy. Let $x, A, B, \varepsilon, \delta$ be chosen as in Definition \ref{d11}. Form $y = x + 1_{\varepsilon A} + 1_{\delta B}$. We have
$$\|x+1_{\varepsilon A}\|\ =\ \|y - P_B(y)\|\ \leqslant\ \mathbf C_{a, f}\widetilde{\sigma}_{f(|B|)}(y)\ \leqslant\ \mathbf C_{a, f}\|x+1_{\delta B}\|.$$

ii) Assume that $(e_n)_{n=1}^\infty$ is $\mathbf C_\ell$-suppression quasi-greedy and has $\Delta_{A, f}$-Property (A, $f$). Let $x\in X$, $m\in\mathbb{N}$, $\Lambda\in G(x, m)$, and a set $B\subset\mathbb{N}$ with $|B|\leqslant f(|\Lambda|)$, which implies that $|B\backslash \Lambda| \leqslant f(|\Lambda\backslash B|)$ by Proposition \ref{p1} item iii). Let $\varepsilon = (\sgn(e_n^*(x))$ and  $\alpha = \min_{n\in \Lambda}|e_n^*(x)|$. By Proposition \ref{p22} and Theorem \ref{bto}, we have
\begin{align*}
    \|x-P_{\Lambda}(x)\|&\ \leqslant\ \Delta_{A, f}\|x-P_{\Lambda}(x) - P_{B\backslash \Lambda}(x) + \alpha 1_{\varepsilon \Lambda\backslash B}\|\\
    &\ \leqslant\ \Delta_{A, f}\|T_{\alpha}(x-P_{\Lambda}(x) - P_{B\backslash \Lambda}(x) + P_{\Lambda\backslash B}(x))\|\\
    &\ \leqslant\ \Delta_{A, f}\mathbf{C}_\ell\|x-P_B(x)\|.
\end{align*}
Since $B$ is arbitrary, we know that $(e_n)_{n=1}^\infty$ is $\Delta_{A, f}\mathbf C_\ell$-$f$-almost greedy. 
\end{proof}

\begin{proof}[Proof of Theorem \ref{mm2}]
i) Assume that $(e_n)_{n=1}^\infty$ is $\mathbf C_{g,f}$-$f$-greedy. First, we show that it is $\mathbf C_{g,f}$-suppression $f$-unconditional. Let $x\in X$, $A\in G(x, f)$, and $(E,F)$ be an $f$-greedy decomposition of $A$. Choose $\alpha > 0$ sufficiently large such that $y:= x + \alpha 1_E$ has $E$ as a greedy set. It follows that $A$ is a greedy set of $y$. We have
$$\|x-P_A(x)\|\ =\ \|y-P_A(y)\|\ \leqslant\ \mathbf C_{g, f}\sigma_{f(|A|)}(y)\ \leqslant\ \mathbf C_{g,f}\|y - \alpha 1_E\|\ =\ \mathbf C_{g,f}\|x\|.$$
Therefore, $(e_n)_{n=1}^\infty$ is $\mathbf C_{g,f}$-suppression $f$-unconditional. To see that $(e_n)_{n=1}^\infty$ has $\mathbf C_{g,f}$-Property (A, $f$), we use a similar argument as in the proof of Theorem \ref{m1} item i).

ii) Assume that $(e_n)_{n=1}^\infty$ is $\mathbf K_{su}$-suppression $f$-unconditional and has $\Delta_{A, f}$-Property (A, $f$). Let $x\in X$, $\Lambda\in G(x, m)$, $A\in\mathbb{N}^{<\infty}$ with $|A|\leqslant f(m)$, and $(a_n)_{n\in A}$ be arbitrary scalars. Set $\varepsilon = (\sgn(e_n^*(x)))$ and $\alpha = \min_{n\in \Lambda}|e_n^*(x)|$. Note that $|A\backslash \Lambda|\ \leqslant\ f(|\Lambda\backslash A|)$. By Proposition \ref{p22}, we have 
\begin{align}\label{ee3}\|x-P_{\Lambda}(x)\|&\ \leqslant\ \Delta_{A, f}\|x-P_{\Lambda}(x) - P_{A\backslash \Lambda}(x) + \alpha 1_{\varepsilon \Lambda\backslash A}\|\nonumber\\
&\ =\ \Delta_{A, f}\|x-P_{\Lambda\cup A}(x) + \alpha 1_{\varepsilon \Lambda\backslash A}\|.
\end{align}
On the other hand, write
$$y\ :=\ x-\sum_{n\in A}a_ne_n\ =\ x-P_{\Lambda\cup A}(x) +\sum_{n\in A}(e_n^*(x)-a_n)e_n + P_{\Lambda\backslash A}(x).$$
Observe that $\Lambda\cup A\in G(y, f)$ because $f(|\Lambda\cup A|)\geqslant f(m)\geqslant |A|$ and 
$$|e_n^*(y)| \ \geqslant\  \alpha \ \geqslant\ \|x-P_{\Lambda\cup A}(x)\|_\infty \ =\ \|y- P_{\Lambda\cup A}(y)\|_\infty, \forall n\in \Lambda\backslash A.$$
Therefore, 
\begin{equation}\label{ee1}\|x-P_{\Lambda\cup A}(x)\|\ =\ \|y-P_{\Lambda\cup A}(y)\|\ \leqslant\ \mathbf K_{su}\|y\|.\end{equation}
On the other hand, by Proposition \ref{pp1} and Theorem \ref{bto}, we get
\begin{align}\label{ee4}
    \|y\|&\ \geqslant\  \frac{1}{\mathbf K_{su}}\left\|x-P_{\Lambda\cup A}(x) +\sum_{n\in A}T_\alpha(e_n^*(x)-a_n)e_n +  \alpha 1_{\varepsilon \Lambda\backslash A}\right\|\nonumber\\
    &\ \geqslant\ \frac{1}{\mathbf K_{su}(\mathbf K_{su}+1)} \left\|\alpha 1_{\varepsilon \Lambda\backslash A}\right\|,
\end{align}
where the last inequality is due to $\Lambda\backslash A$ being a greedy set of $$x-P_{\Lambda\cup A}(x) +\sum_{n\in A}T_\alpha(e_n^*(x)-a_n)e_n +  \alpha 1_{\varepsilon \Lambda\backslash A}.$$
From \eqref{ee1} and \eqref{ee4}, we obtain
$$\|x - P_A(x) + \alpha 1_{\varepsilon \Lambda\backslash A}\|\ \leqslant\ (\mathbf K^2_{su} + 2\mathbf K_{su})\|y\|,$$
which, combined with \eqref{ee3}, gives
$$\|x-P_{\Lambda}(x)\|\ \leqslant\ \Delta_{A, f}(\mathbf K^2_{su} + 2\mathbf K_{su})\left\|x-\sum_{n\in A}a_ne_n\right\|,$$
as desired. 
\end{proof}

We are now ready to prove the equivalence between $f$-almost greedy and $f$-greedy bases for each non-identity $f\in\mathcal{F}$.

\begin{proof}[Proof of Theorem \ref{m3}]
We need only to show that an $f$-almost greedy basis is $f$-greedy. Assume that $(e_n)_{n=1}^\infty$ is $f$-almost greedy; that is, $(e_n)_{n=1}^\infty$ is $\mathbf C_q$-quasi-greedy and $\Delta$-$f$-superdemocratic according to Theorem \ref{m1} and Proposition \ref{pe2}.

We know that $\lim_{x\rightarrow\infty}\frac{f(x)}{x} = \mathbf c$ exists by Proposition \ref{p1}. Since $f$ is non-identity, the proof of Proposition \ref{p1} shows that $\mathbf c < 1$. Suppose first that $0 < \mathbf c < 1$. Since $f(x)/x$ is decreasing, we know that
$\mathbf c x \leqslant f(x)$ for all $x > 0$.
By Theorem \ref{dkktc}, $(e_n)_{n=1}^\infty$ is almost greedy. We use Proposition \ref{ptu} to conclude that the basis is $f$-greedy. 

For the rest of the proof, we assume that $\lim_{x\rightarrow\infty} \frac{f(x)}{x} = 0$. By Theorem \ref{mm2} and Proposition \ref{ps}, it suffices to verify that $(e_n)_{n=1}^\infty$ has Property ($f$). Choose $x, A, B, \varepsilon$ as in Definition \ref{d20}.
We want to show that there exists a constant $\mathbf C$ (independent of $x, A, B, \varepsilon$) such that 
$$\|P_A(x)\|\ \leqslant\ \mathbf C\|x + 1_{\varepsilon B}\|.$$
Since $f$ is concave, we have
$$|A|\ \leqslant\  f(|A|+|B|)\ \leqslant\ f(|A|) + f(|B|).$$
Choose $N \geqslant 6$ be such that $f(x) \leqslant x/3$ for all $x\geqslant N$. 
If $|A|\leqslant N$, then
$$\|P_A(x)\|\ \leqslant\ \|P_A\|\|x\|\ \leqslant\ \mathbf c_2^2\|x\|\ \leqslant\ \mathbf c_2^2 (\mathbf C_q+1)\|x+1_{\varepsilon B}\|.$$
If $|A| > N$, then
$$f(|B|) \ \geqslant\ |A| - f(|A|)\ \geqslant\ \frac{2|A|}{3}.$$
Partition $A$ into two sets $A_1, A_2$ such that 
$$|A_i|\ \leqslant\ \frac{|A|}{2}+1\ \leqslant\ \frac{2|A|}{3}\ \leqslant\ f(|B|).$$
By $\Delta$-$f$-superdemocracy, we get
$$\sup_{\delta}\|1_{\delta A}\|\ \leqslant\ \sup_{\delta'}\|1_{\delta' A_1}\| + \sup_{\delta''}\|1_{\delta'' A_2}\|\ \leqslant\ 2\Delta\|1_{\varepsilon B}\|.$$
Therefore, by convexity and the quasi-greedy property, 
$$\|P_A(x)\|\ \leqslant\ \sup_{\delta}\|1_{\delta A}\|\ \leqslant\ 2\Delta\|1_{\varepsilon B}\|\ \leqslant\ 2\Delta\mathbf C_q\|x+ 1_{\varepsilon B}\|.$$
We conclude that $(e_n)_{n=1}^\infty$ has Property ($f$), which completes our proof. 
\end{proof}

\begin{proof}[Proof of Corollary \ref{c1}]
The equivalence between i) and ii) is due to Theorem \ref{m3}. The equivalence between ii) and iii) is due to Theorem \ref{m1}. According to Proposition \ref{p2}, iii) $\Longleftrightarrow$ iv). Next, Proposition \ref{pe2} gives the equivalence among iii), iv), and v). Finally, the second statement follows from the fact that $f$-unconditionality implies quasi-greediess and Theorem \ref{mm2}.
\end{proof}

We shall end this section with an easy corollary that shall be used to prove Theorem \ref{m2}.
\begin{cor}\label{c4}
Let $g\in \mathcal{F}$ and $g$ is bounded. Then a basis is $g$-almost greedy if and only if it is quasi-greedy. 
\end{cor}

\begin{proof}
By Corollary \ref{c1}, it suffices to show that if $g$ is bounded, then every basis $(e_n)_{n=1}^\infty$ is $g$-democratic. Indeed, let $A, B\in \mathbb{N}^{<\infty}$ with $|A|\leqslant g(|B|)$. If $|B| = 0$, then $|A| = 0$, and there is nothing to prove. Assume that $|B|\geqslant 1$ and let $\sup_{m}g(m) = s$. We have
$$\|1_A\|\ \leqslant\ |A|\mathbf c_2\ \leqslant\ s\mathbf c_2\ \leqslant\ s\mathbf c_2\sup_{n}\|e_n^*\|\|1_B\|\ \leqslant\ s\mathbf c^2_2\|1_B\|.$$
Hence, $(e_n)_{n=1}^\infty$ is $g$-democratic, as desired. 
\end{proof}


\section{When $f$ varies - A modification of the Schreier space}\label{Schreier}

The main goal of this section is to study relations among $f$-(almost) greedy bases when $f$ varies. 

\begin{prop}\label{ce2}
Let $f, g\in \mathcal{F}$ be such that there exists a constant $\mathbf c > 0$ satisfying $f(m)\geqslant \mathbf c g(m)$ for all $m\in\mathbb{N}$. If $(e_n)_{n=1}^\infty$ is $f$-almost greedy, then it is $g$-almost greedy.
\end{prop}

\begin{proof} Let $f, g\in \mathcal{F}$ such that there exists $\mathbf c > 0$ satisfying $f(m)\geqslant \mathbf cg(m)$ for all $m\in\mathbb{N}$.
Due to Corollary \ref{c1}, it suffices to prove that a $\Delta$-$f$-democratic basis is $g$-democratic. Let $A, B\in \mathbb{N}^{<\infty}$ with $|A|\leqslant g(|B|)$. We have $$|A|\ \leqslant\ g(|B|)\ \leqslant\ \frac{1}{\mathbf c}f(|B|).$$ 
If $\frac{1}{\mathbf c}\leqslant 1$, then by $f$-democracy, we obtain 
$\|1_A\|\leqslant \Delta\|1_B\|$. Suppose that $\frac{1}{\mathbf c} > 1$. We consider two cases.

Case 1: $f(|B|) \leqslant 2$. Then $|A|\leqslant \frac{2}{\mathbf c}$. We have
$$\|1_A\|\ \leqslant\ |A|\mathbf c_2\ \leqslant\ \frac{2}{\mathbf c}\mathbf c_2 \ \leqslant\ \frac{2}{\mathbf c}\mathbf c_2\sup_{n}\|e_n^*\|\|1_B\|\ \leqslant\ \frac{2\mathbf c_2^2}{\mathbf c}\|1_B\|.$$

Case 2: $f(|B|) > 2$. Partition $A$ into $s:= \lceil \frac{2}{\mathbf c}\rceil$ disjoint subsets $A_1, \ldots, A_s$ such that for all $1\leqslant i\leqslant s$,  
$$|A_i| \ \leqslant\ \frac{|A|}{s} + 1\ \leqslant\ \frac{\mathbf c|A|}{2} + 1\ \leqslant\ \frac{f(|B|)}{2} + 1\ <\ f(|B|).$$
By $f$-democracy, we get
$$\|1_{A_i}\|\ \leqslant\ \Delta\|1_B\|, \forall 1\leqslant i\leqslant s.$$
Therefore, 
$$\|1_A\|\ \leqslant\ s\Delta\|1_B\|.$$

In general, we have shown that
$$\|1_A\|\ \leqslant\ \max\left\{\frac{2\mathbf c_2^2}{\mathbf c}, \left\lceil \frac{2}{\mathbf c}\right\rceil\Delta\right\}\|1_B\|,$$
and hence, $(e_n)_{n=1}^\infty$ is $g$-democratic.
\end{proof}

\begin{cor}\label{c22}
Let $f, g\in \mathcal{F}$ be such that $0 < \inf_{m\in\mathbb{N}} \frac{f(m)}{g(m)} \leqslant \sup_{m\in \mathbb{N}} \frac{f(m)}{g(m)} < \infty$, then a basis is $f$-almost greedy if and only if it is $g$-almost greedy. 
\end{cor}

\begin{proof}
Use Proposition \ref{ce2}.
\end{proof}

We now turn our attention to a notable example (in \cite{BDKOW}) of a PG basis that is not almost greedy. We shall modify the basis there to prove the following theorem
\begin{thm}\label{m2}
i) Let $f, g\in \mathcal{F}$ such that $\sup_{m\in \mathbb{N}}\frac{f(m)}{g(m)} = \infty$. Then there exists a $1$-unconditional basis that is $g$-greedy but is not $f$-almost greedy. 

ii) There exists an $1$-unconditional basis that is not $f$-almost greedy for all unbounded $f\in \mathcal{F}$.
\end{thm}

Let us generalize the construction in \cite{BDKOW} to any function $f\in \mathcal{F}$. Define $$\mathcal{A} = \left\{A\in\mathbb{N}^{<\infty}: A\neq\emptyset \mbox{ and }\frac{1}{f(1)}f(\min A) \geqslant |A|\right\}$$ (thus all singletons are in $\mathcal{A}$) and $X_f$ be the completion of $c_{00}$ under the following norm: for $x = (x_1, x_2, \ldots)$, 
$$\|x\|\ =\ \sup_{A\in \mathcal{A}}\sum_{n\in A}|x_n|.$$
Let $(e_n)_{n=1}^\infty$ be the normalized canonical basis, which is clearly $1$-unconditional.

\begin{lem}\label{l2}
For all $N\in\mathbb{N}$, it holds that
$$\frac{1}{2} f(N)-1\ \leqslant\ \left\|\sum_{n=1}^N e_n\right\| \ \leqslant\ \frac{1}{f(1)} f(N).$$
\end{lem}

\begin{proof}
Let $A\subset \{1, 2, \ldots, N\}$ and $A\in \mathcal{A}$. Then $\min A\leqslant N$. Because $f$ is increasing, 
$$\sum_{m\in A}\left|e_m^*\left(\sum_{n=1}^N e_n\right)\right|\ =\ |A|\ \leqslant\  \frac{1}{f(1)}f(\min A)\ \leqslant\ \frac{1}{f(1)}f(N).$$
We conclude that $\left\|\sum_{n=1}^N e_n\right\| \ \leqslant\ \frac{1}{f(1)} f(N)$.

Next, we establish a lower bound. Choose a nonnegative integer $k\leqslant N-1$ to be the smallest such that $A=\{k+1, \ldots, k+(N-k)\}\in \mathcal{A}$, i.e., $f(k+1)\geqslant f(1)(N-k)$. If $k = 0$, then $|A| = N\geqslant f(N)$. Suppose that $k\geqslant 1$. Then $f(k) < f(1)(N-k+1)$. We consider two cases.

Case 1: if $k > N/2$, then 
$$|A|\ =\ N-k \ >\ \frac{f(k)}{f(1)}-1\ \geqslant\ \frac{f(N/2)}{f(1)}-1\ \geqslant\ \frac{1}{2f(1)}f(N)-1.$$

Case 2: if $k \leqslant N/2$, then $N\geqslant 2$ because $k\geqslant 1$. By Proposition \ref{p1} item ii),
$$|A|\ = \ N-k\ \geqslant \ \frac{N}{2}\ \geqslant \ f\left(\frac{N}{2}\right)\ \geqslant\ \frac{1}{2}f(N).$$

Therefore, $\left\|\sum_{n=1}^N e_n\right\| \ \geqslant\ \frac{1}{2} f(N)-1$.
\end{proof}

\begin{proof}[Proof of Theorem \ref{m2}]
We prove item i). Let $f, g\in \mathcal{F}$ such that $\sup_{m\in \mathbb{N}}\frac{f(m)}{g(m)} = \infty$. Clearly, $f$ is unbounded. We first assume that $g$ is also unbounded. We claim that the canonical basis of $X_{g}$ is $g$-greedy but is not $f$-almost greedy. 

\underline{$(e_n)_{n=1}^\infty$ is $g$-greedy}: let $A, B\in\mathbb{N}^{<\infty}$ with $|A|\leqslant g(|B|)$. By Lemma \ref{l2}, we get 
$$\|1_B\|\ \geqslant\ \left\|\sum_{n=1}^{|B|}e_n\right\|\ \geqslant\ \frac{1}{2}g(|B|)-1\ \geqslant\ \frac{1}{2}|A|-1\ \geqslant\ \frac{1}{2}\|1_A\|-1.$$
Hence, 
$$\frac{1}{2}\|1_A\|\ \leqslant\ \|1_B\| + 1\ \leqslant\ \|1_B\| + \sup_{n}\|e_n^*\|\|1_B\|\ =\ (\mathbf c_2 + 1)\|1_B\|.$$
Therefore, $(e_n)_{n=1}^\infty$ is $g$-democratic. By unconditionality and Corollary \ref{c1}, $(e_n)_{n=1}^\infty$ is $g$-greedy.

\underline{$(e_n)_{n=1}^\infty$ is not $f$-almost greedy}: let $m\in \mathbb{N}$, $B = \{1, 2, \ldots, m\}$ and $A\subset\mathbb{N}^{<\infty}$ such that $|A| = \lfloor f(m) \rfloor$ and $\frac{1}{g(1)}g(\min A) \geqslant |A|$ (which is possible since $\lim_{x\rightarrow\infty} g(x) = \infty$). By Lemma \ref{l2}, 
$$\|1_B\| \ \leqslant\ \frac{1}{f(1)}g(m)\mbox{ and }\|1_A\| \ =\ \lfloor f(m) \rfloor,$$
which implies that
$$\frac{\|1_A\|}{\|1_B\|}\ \geqslant \ f(1)\frac{\lfloor f(m) \rfloor}{g(m)}.$$
Since $\sup_{m\in \mathbb{N}}\frac{f(m)}{g(m)} = \infty$, we do not have $f$-democracy. By Theorem \ref{m1} and Proposition \ref{pe2}, $(e_n)_{n=1}^\infty$ is not $f$-almost greedy.

In the case when $g$ is bounded, define $h = \sqrt{f}\in \mathcal{F}$ because $\mathcal{F}$ is closed under compositions. Note that $h$ is unbounded. Using the same argument as above, the canonical basis of $X_h$ is $h$-greedy but is not $f$-almost greedy. By the proof of Corollary \ref{c4}, the canonical basis of $X_h$ is $g$-democratic, and according to Corollary \ref{c1}, the basis is $g$-greedy.

Next, we prove item ii). Consider the space $(\ell_1 \oplus c_{0})_{\ell_1}$ with the basis $\mathcal{X} = (x_n)_{n=1}^\infty$ defined as
$$x_n \ =\ \begin{cases}e_{k}&\mbox{ if } n = 2k-1,\\ f_k &\mbox{ if }n = 2k.\end{cases}$$
where $(e_n)_{n=1}^\infty$ and $(f_n)_{n=1}^\infty$ are the canonical bases of $\ell_1$ and $c_0$, respectively. Clearly, $\mathcal{X}$ is $1$-unconditional. Let $f\in \mathcal{F}$ be unbounded. Then $\mathcal{X}$ is not $f$-disjoint democratic. Indeed, pick $m\in\mathbb{N}$, $A = \{1, 3, \ldots, 2\lfloor f(m)\rfloor-1\}$, and $B = \{2, 4, \ldots, 2m\}$. It follows that
$\frac{\|1_A\|}{\|1_B\|} = \lfloor f(m)\rfloor\rightarrow\infty$ as $m\rightarrow\infty$. Hence, $\mathcal{X}$ is not $f$-almost greedy.
\end{proof}

\begin{rek}\normalfont We consider the PG basis in \cite[Proposition 6.10]{BDKOW}, which is the canonical basis of $X_f$, when $f$ is the square-root function. 
According the proof of Theorem \ref{m2}, we know that the basis is $f$-greedy though it is not almost greedy. Consequently, though some PG bases are not almost greedy, they can be $f$-greedy for a well-chosen $f$. Since the notion of being $f$-greedy is order-free, this provides a partial solution to the concern of some researchers regarding the order-dependent nature of PG bases. 
\end{rek}

\begin{cor}
Let $f, g\in \mathcal{F}$. 
\begin{enumerate}
    \item[i)] If $\sup_{m\in\mathbb{N}} \frac{f(m)}{g(m)} < \infty$, then the canonical basis of $X_g$ is $f$-greedy.
    \item[ii)] If $\sup_{m\in \mathbb{N}}\frac{f(m)}{g(m)} = \infty$ and $g$ is unbounded, then the canonical basis of $X_g$ is not $f$-almost greedy.
    \item[iii)] If $g$ is bounded, then the canonical basis of $X_g$ is equivalent to the canonical basis of $c_0$. In this case, the canonical basis of $X_g$ is $f$-greedy for all $f\in \mathcal{F}$.
\end{enumerate}
\end{cor}

\begin{proof}
i) Suppose that $\sup_{m\in\mathbb{N}} \frac{f(m)}{g(m)} =: \mathbf c < \infty$. Then $f(m)\leqslant \mathbf c g(m), \forall m\in\mathbb{N}$. Since the canonical basis $(e_n)_{n=1}^\infty$ of $X_g$ is $g$-greedy according to the proof of Theorem \ref{m2}, Proposition \ref{ce2} states that $(e_n)_{n=1}^\infty$ is $f$-democratic. By unconditionality, Theorem \ref{mm2}, and Proposition \ref{pe2}, $(e_n)_{n=1}^\infty$ is $f$-greedy.

ii) Use the same reasoning as in the proof of Theorem \ref{m2}.

iii) Suppose that $\sup_{x\geqslant 0} g(x) = \mathbf c$. Let $(e_n)_{n=1}^\infty$ be the canonical basis of $X_g$. For all $A\in\mathbb{N}^{<\infty}$ and sign $\varepsilon$, we have $\|1_{\varepsilon A}\|\leqslant \frac{\mathbf c}{g(1)}$. Hence, $(e_n)_{n=1}^\infty$ is equivalent to the canonical basis of $c_0$.
\end{proof}

\section{Examples}\label{ex}

\subsection{The (normalized) multivariate Haar basis} 

Consider the multivariate Haar basis $\mathcal{H}^d_p$ of $L^p([0,1]^d)$ for $1 < p < \infty$. A famous result by Temlyakov \cite{T} states that the univariate $\mathcal{H}^{1}_p$ is a greedy basis of $L^p[0,1]$. However, the same conclusion does not hold for the general multivariate basis $\mathcal{H}_p^d = \mathcal{H}_p\times \cdots \mathcal{H}_p$. While $\mathcal{H}_p^d = \{h_n\}$ is an unconditional basis for $L^p([0,1]^d)$, the system is not democratic and thus, not greedy \cite[(1.3)]{T2}. However, one has the following result from \cite{KPT} (see also \cite[Prosposition 7.6]{DGHKT}).

\begin{prop}\label{p101}
For any $A\in\mathbb{N}^{<\infty}$ with $|A| = m\geqslant 2$, there exist $\mathbf C_1, \mathbf C_2 > 0$ such that if $p\in [2, \infty)$,
$$\mathbf C_1 m^{1/p}\ \leqslant\ \left\|\sum_{n\in A}h_n\right\| \ \leqslant\ \mathbf C_2 (\ln m)^{h(p, d)}m^{1/p},$$
while if $p\in (1, 2)$,
$$\mathbf C_1 m^{1/p}(\ln m)^{-h(p, d)}\ \leqslant\ \left\|\sum_{n\in A}h_n\right\|\ \leqslant\ \mathbf C_2 m^{1/p},$$
where $h(p, d):= (d-1)\left|\frac{1}{2}-\frac{1}{p}\right|$.
\end{prop}

\begin{thm}
For any $\varepsilon > 0$, let $f(x) = x^{1-\varepsilon}$. The $d$-multivariate system $\mathcal{H}_p^d = \{h_n\}$ is $f$-greedy.
\end{thm}

\begin{proof}
By Corollary \ref{c1}, it suffices to show that $\mathcal{H}_p^d$ is $f$-democratic. Let $A, B\in \mathbb{N}^{<\infty}$ such that $|A|\leqslant |B|^{1-\varepsilon}$. Also let $\mathbf C_1, \mathbf C_2$ be as in Proposition \ref{p101}. Choose $N\geqslant 2$ to be sufficiently large such that
$$x^{\varepsilon/p}\ \geqslant\ \frac{\mathbf C_2}{\mathbf C_1}(\ln x)^{h(p, d)}, \forall x\geqslant N.$$
Without loss of generality, assume $|B|^{1-\varepsilon}\geqslant |A|\geqslant N$. (If $|A|\leqslant N$, we use semi-normalization of the basis).
We consider only the case $p\in [2, \infty)$ as the case $p\in (1, 2)$ is similar. We have 
\begin{align*}\left\|1_A\right\|\ \leqslant\ \mathbf C_2 (\ln |A|)^{h(p, d)}|A|^{1/p}&\ \leqslant\ \mathbf C_2(\ln |B|)^{h(p, d)}|B|^{(1-\varepsilon)/p}\\
&\ \leqslant\ (\mathbf C_1|B|^{\varepsilon/p})|B|^{(1-\varepsilon)/p}\ =\ \mathbf C_1|B|^{1/p}\ \leqslant\ \|1_B\|.
\end{align*}
Therefore, $\mathcal{H}_p^d$ is $f$-democratic, as desired. 
\end{proof}

\subsection{A conditional basis that is $x/\lambda$-unconditional for any $\lambda > 1$}\label{ex1c}
We borrow the conditional, QG basis $\mathcal{B}$ by Konyagin and Temlyakov \cite{KT1}, who considered the canonical basis $\mathcal{B}$ of the completion of $c_{00}$ with respect to the norm
$$\|(a_n)_{n=1}^\infty\|\ :=\ \max\left\{\left(\sum_{n}|a_n|^2\right)^{1/2}, \sup_{N}\left|\sum_{n=1}^N\frac{a_n}{\sqrt{n}}\right|\right\}.$$

Fix $\lambda > 1$ and let us show that the basis is $x/\lambda$-unconditional. By Proposition \ref{ps}, it suffices to show that $\mathcal{B}$ has Property ($x/\lambda$). Choose $x, A, B, \varepsilon$ as in Definition \ref{d20}. Then $|A|\leqslant |B|/(\lambda-1)$. Clearly, $\|x+1_{\varepsilon B}\|\geqslant |B|^{1/2}$. We have
\begin{align*}\|P_A(x)\|&\ =\ \max\left\{\left(\sum_{n\in A}|e_n^*(x)|^2\right)^{1/2}, \sup_{N}\left|\sum_{\substack{1\leqslant n\leqslant N\\n\in A}}\frac{e_n^*(x)}{\sqrt{n}}\right|\right\}\\
&\ \leqslant\ \max\left\{|A|^{1/2}, \sum_{n=1}^{|A|}\frac{1}{\sqrt{n}}\right\}\ \leqslant\ 2|A|^{1/2}\\
&\ \leqslant\ 2(\lambda-1)^{-1/2}|B|^{1/2}\ \leqslant\ 2(\lambda-1)^{-1/2}\|x+1_{\varepsilon B}\|.
\end{align*}
Hence, $\mathcal{B}$ is $x/\lambda$-unconditional, as claimed.

\subsection{A non-QG basis that has Property ($x/\lambda$) for any $\lambda > 1$}\label{ex1b}
We use the example in \cite[Section 5.3]{C3}.
Let us first construct the weight sequence $\omega = (w_n)_{n=1}^\infty$.
For $n\geqslant 1$, let $t_n = \frac{1}{\sqrt{n}}$, $L_n = e^{(\ln n)^2}$, and $a_n = \frac{1}{\sqrt{n}\ln (n+1)}$. Define an increasing sequence $(N_n)_{n=0}^\infty$ recursively. Set $N_0 = 0$. Choose $N_1>10$ to be the smallest such that 
$$b_1 \ :=\ a_1t_1\left(\sum_{n=1}^{N_1} t_n\right)^{-1} \ <\ \frac{a_1}{L_1}.$$
Once $N_j$ and $b_j$ are defined for $j\geqslant 1$, choose $N_{j+1}$ to be the smallest such that 
\begin{equation}\label{ee2}b_{j+1}\ :=\ a_{j+1}t_{j+1}\left(\sum_{n=N_j+1}^{N_j+ N_{j+1}} t_n\right)^{-1}\ <\ \min\left\{\frac{a_{j+1}}{L_{j+1}}, b_j\right\}\mbox{ and }\frac{N_{j+1}}{N_j} \ >\ 10.\end{equation}
For $j\geqslant 1$, denote the finite sequence $(t_{n})_{n= N_{j-1}+1}^{N_{j-1}+N_j} = (t_{N_{j-1}+1}, \ldots, t_{N_{j-1}+N_j})$ by $B_j$.

We now define a weight $\omega:=(w_n)_{n=1}^\infty$ on $\mathbb{N}$ as follows:
$$(w_n) \ =\ (t_1, B_1, t_2, B_2, t_3, B_4, t_4, \ldots).$$
In words, the weight $(w_n)$ is chosen such that the first one is $t_1$; the next $N_1$ weights are taken from $B_1$ in the same order; the next weight is $t_2$, then the next $N_2$ weights are taken from $B_2$ in the same order, and so on.

Let $X$ be the completion of $c_{00}$ under the following norm 
$$\left\|\sum_{n}x_n e_n\right\|\ =\ \max\left\{\sup_{N}\left|\sum_{n=N}^\infty w_nx_n\right|, \left(\sum_{n=1}^\infty |x_n|^2\right)^{1/2}\right\}.$$
For ease of notation, set 
$$\left\|\sum_{n}x_n e_n\right\|_1\: = \ \sup_{N}\left|\sum_{n=N}^\infty w_nx_n\right|\mbox{ and }\left\|\sum_{n}x_n e_n\right\|_2\ =\ \left(\sum_{n=1}^\infty |x_n|^2\right)^{1/2}.$$
Clearly, the standard unit vector basis $\mathcal B$ is Schauder and normalized in $(X, \|\cdot\|)$. As shown in \cite{C3}, the basis is not quasi-greedy. 

We shall show that it has Property ($x/\lambda$) for any $\lambda > 1$. Fix $\lambda > 1$ and let $x, A, B, \varepsilon$ be chosen as in Definition \ref{d20}. Then $|A|\leqslant (\lambda-1)^{-1}|B|$. We know that $\|x+1_{\varepsilon B}\|\geqslant |B|^{1/2}$. On the other hand,
$$\|P_A(x)\|\ \leqslant\ \max\left\{\sum_{\substack{n\geqslant 1\\n\in A}} w_n, |A|^{1/2}\right\}\ \leqslant\ \max\left\{2\sum_{n=1}^{|A|}\frac{1}{\sqrt{n}}, |A|^{1/2}\right\}\ \leqslant\ 4|A|^{1/2}.$$
Hence, 
$$\|P_A(x)\|\ \leqslant\ 4|A|^{1/2}\ \leqslant\ 4(\lambda-1)^{-1/2}|B|^{1/2}\ \leqslant\ 4(\lambda-1)^{-1/2}\|x+1_{\varepsilon B}\|.$$
We conclude that $\mathcal{B}$ has Property ($x/\lambda$).

\subsection{A quasi-greedy basis that does not have Property $(f)$ for any unbounded $f$}\label{ex1a}
Let $(x_n)_{n=1}^\infty$ be a Schauder, conditional, QG basis of a Banach space $X$ such that $\|x\|_{X}\leqslant 1$ implies that $|x_n^*(x)|\leqslant 1$ for all $n$. We know such a basis exists due to the Konyagin and Temlyakov's example mentioned in Example \ref{ex1c}. Equip $Z = X \oplus c_0$ with the norm 
$$\|(x, v)\|_{Z} = \max\{\|x\|_{X}, \|v\|_{c_0}\}, \forall x\in X, \forall v\in c_0.$$
Let $\mathcal{Z} = (z_n)_{n=1}^\infty$ be the basis of $Z$ defined as
$$z_{n} \ =\ \begin{cases}(x_{k}, 0)&\mbox{ if }n = 2k-1,\\ (0, e_{k}) &\mbox{ if }n = 2k,\end{cases}$$
where $(e_n)_{n=1}^\infty$ is the canonical basis of $c_0$. Clearly, $\mathcal{Z}$ is QG because the $c_0$-direct sum of two QG bases is again QG. 
However, the basis $\mathcal{Z}$ is not $f$-unconditional for any unbounded $f\in\mathcal{G}$. Indeed, pick $L > 1$.
Since $(x_n)_{n=1}^\infty$ is conditional, there exists a finitely supported vector $x = \sum_{n=1}^N a_nx_n$ with $\|x\|_{X} = 1$ (and thus, $|a_n|\leqslant 1$) and a finite set $A\subset\supp(x)$ such that 
$\|P_A(x)\|_{X}\ >\ L$. Choose $m$ such that $f(m) > |A|$. Consider the following vector in $Z$
$$z\ :=\ \sum_{j=1}^N a_{j}z_{2j-1} + \sum_{n = 1}^m z_{2n}.$$
Let $A' = \{2j-1: j\in A\}$ and $B = \{2, 4, \ldots, 2m\}$. Observe that $B$ is a greedy set of $z$ and 
$$|A'\cup B| - f(|A'\cup B|)\ \leqslant\ |A| + |B| - f(m)\ <\ |B|.$$
However, 
$$\|P_{A'\cup B}(z)\|_{Z}\ =\ \|P_A(x)\|_{X}\ >\ L\ =\ L\|z\|_{Z}.$$
Since $L > 1$ is arbitrary, our basis $\mathcal{Z}$ is not $f$-unconditional.


\ \\
\end{document}